\documentclass[12pt,reqno]{amsart}

\usepackage{amsmath,amssymb,amsfonts}
\usepackage[frame,cmtip,arrow,matrix,line,graph,curve]{xy}
\usepackage{graphpap,color,paralist,pstricks}
\usepackage[mathscr]{eucal}
\usepackage[pdftex]{graphicx}
\usepackage[pdftex,colorlinks,backref=page,citecolor=blue]{hyperref}
\usepackage{tikz}
\usepackage{kotex}
\usepackage{array}
\usepackage{pgfplots}
\usepackage{mathrsfs}
\usetikzlibrary{arrows} 
\usepackage{tikz-cd}
\usepackage{mathtools}

\newcolumntype{P}[1]{>{\centering\arraybackslash}p{#1}}
\newcolumntype{M}[1]{>{\centering\arraybackslash}m{#1}}

\newtheorem{theorem}{Theorem}[section]
\newtheorem{proposition}[theorem]{Proposition}

\newtheorem{lemma}[theorem]{Lemma}

\theoremstyle{definition}

\newtheorem{remark}[theorem]{Remark}

\newtheorem{notation}[theorem]{Notation}

\newcommand{\PP}{\mathbb{P}}

\newcommand{\FF}{\mathbb{F}}

\newcommand{\cO}{\mathcal{O} }

\newcommand{\enm}[1]{\ensuremath{#1}}

\newcommand{\cal}[1]{\mathcal{#1}}

\renewcommand{\bar}[1]{\overline{#1}}

\newcommand{\Ii}{\enm{\cal{I}}}

\newcommand{\Oo}{\enm{\cal{O}}}

\newcommand\bH{\mathbf{H}}

\def\Sym{\mathrm{Sym} }

\def\Gr{\mathrm{Gr} }

\hypersetup{%
pdftitle={Rational quartic curves in the quintic del Pezzo threefold},%
pdfauthor={Kiryong Chung, Jeong-Seop Kim},%
pdfkeywords={},%
citecolor=blue,%
linkcolor=blue,%
}

\begin{document}

\title{A remark on rational quartic curves in prime Fano threefolds of degree $22$}
\date{\today}

\author{Kiryong Chung}
\address[Kiryong Chung]
{Department of Mathematics Education, Kyungpook National University, 80 Daehakro, Bukgu, Daegu 41566, Korea}
\email{krchung@knu.ac.kr}

\author{Jeong-Seop Kim}
\address[Jeong-Seop Kim]
{Department of Mathematics Education, Sunchon National University, 255 Jungang-ro, Suncheon-si, Jeollanam-do 57922, Korea}
\email{jeongseop@scnu.ac.kr}

\subjclass[2020]{14J45, 14E15, 14M15, 14C05.}
\keywords{Rational quartic curves, Prime Fano threefolds, Hilbert scheme}

\begin{abstract}
In this short note, using the Sarkisov link between a prime Fano threefold $V_{22}$ of degree~$22$ and the quintic del Pezzo threefold $V_5$, we prove that the Hilbert scheme of rational quartic curves in $V_{22}$ admits a generically $2$-to-$1$ rational map onto the projective space $\PP^4$.
\end{abstract}

\maketitle

\section{Introduction}
\subsection{Previous works}
The geometry of moduli spaces of rational curves in projective varieties has been extensively studied from various perspectives, including birational geometry and enumerative geometry.
Among these directions, the construction of compactifications and the investigation of their geometric invariants play a fundamental role in understanding the global structure of these moduli spaces.
In \cite{CK11, CHK12, CM17, CHL18, Chu22}, the authors focused on the birational relation among several compactifications of moduli spaces of rational curves.
Through a systematic study of various compactifications, the authors established explicit birational relations by extending the moduli-theoretic birational maps. On the other hand, in \cite{CKK25}, by finding torus-invariant quartic curves in the Mukai–Umemura threefold (\cite{MU83}), the authors proved the smoothness of the corresponding Hilbert scheme. These results subsequently enabled the application of various localization techniques to compute Donaldson--Thomas (DT) type invariants and verify several important conjectures.
For detailed descriptions of conjectures and their proof, see \cite{CMT18, CT21, KP08} and \cite{CLW24, CW26}.
Also, we obtained global descriptions of Hilbert schemes of rational curves for the quadric threefold and the quintic del Pezzo threefold (\cite{CHY23} and \cite{CKK25b}). Once such geometric descriptions are established, localization formulas (\cite{GP99}) can be canonically applied to compute the associated DT-type invariants. On the other hand, for prime Fano threefolds $V_{22}$ of degree $22$, including the Mukai–Umemura variety, the use of exceptional collections (and their mutations) of $V_{22}$ allowed us to identify a component of the Hilbert scheme with a Grassmannian variety (\cite{CL26}).
This approach depends heavily on the algebraic properties of rational curves in $V_{22}$. 
\subsection{Result}
In this short article, we propose a geometric property of the Hilbert scheme of rational quartic curves in $V_{22}$ inspired by so-called \emph{correlators} in quantum cohomology (\cite[Section 2.2]{BM04}).
Let $\bH_4(X)$ be the Hilbert scheme of curves $C$ in a smooth projective variety $X$ with Hilbert polynomial $\chi(\cO_C(m))=4m+1$.
Unless otherwise stated, we henceforth assume that $\bH_4(X)$ is irreducible to simplify the discussion.
Let $V_{22}\dashrightarrow V_5$ be \emph{the Sarkisov link} with centers: a line $L\subset V_{22}$ and a rational quintic curve $\Gamma\subset V_5$. For the definition and geometric property of the Sarkisov link, see Section \ref{susar} and \cite{KPS18}. 
It is well known that the strict transform $\bar{C}$ of a general rational quartic curve $C \subset V_{22}$, not meeting the line $L$, intersects $\Gamma$ in four points (cf. \cite[Section 1.3]{TZ12}).
We prove that such correspondence provides a covering structure of the Hilbert scheme $\bH_4(V_{22})$ over a projective space $\PP^4$.

\begin{theorem}[$=$ Theorem~\ref{double_cover}]\label{mainthm}
Let $\bH_4(V_{22})$ be the Hilbert scheme of rational quartic curves in $V_{22}$.
Then there exists a $2$-to-$1$ rational map
\[
\Phi:\bH_4(V_{22})\dashrightarrow \Sym^4\Gamma (\cong \PP^4)
\]
which associates a general rational quartic curve to four points in a rational quintic curve $\Gamma\subset V_5$.
\end{theorem}

The outline of the proof of Theorem \ref{mainthm} is as follows.
For four general points in $V_5$, there exist three rational quartic curves (Lemma \ref{three_quartics}).
One of them is contained in the strict transform of the flopping center of the Sarkisov link, and thus the inverse image of $\Phi$ consists of two rational quartic curves in $V_{22}$.
Moreover, by analyzing the exceptional locus of a residual map, we can show that the rational map $\Phi$ in Theorem \ref{mainthm} is branched over a hypersurface in $\PP^4$. For a geometric meaning of the branch locus of the map $\Phi$, see the proof of Theorem~\ref{double_cover}.
\subsection*{Acknowledgements}
The authors gratefully acknowledge the many helpful suggestions and comments provided by Sanghyeon Lee.

\section{Preliminaries}
In this section, we review the results in \cite{CKK25b,FN89, Ili94} and \cite{IP99, KPS18}, which will be used for the proof of Theorem \ref{mainthm}.

\subsection{Rational quartic curves in $V_5$}
Let $V_5$ be the quintic del Pezzo threefold.
In \cite{CKK25b}, by extending the result of \cite[Proposition~4.11]{FGP19}, it is proved that the Hilbert scheme $\bH_4(V_5)$ of rational quartic curves in $V_5$ is isomorphic to a Grassmannian bundle over the Hilbert scheme $\bH_1(V_5)$ of lines in $V_5$.
More precisely, a rational quartic curve in $V_5$ arises as the residual curve of a linear section containing a line in $V_5$. In detail, let $l$ be a line in $V_5$, and let $\Lambda$ be a $4$-dimensional linear subspace of $\PP^6$ containing $l$.
Then $\Lambda\cap V_5=l\cup C$ for a rational quartic curve $C$ in $V_5$.
Moreover, such a residual line $l$ is uniquely determined by $C$.
Indeed, if $C$ admitted two distinct residual lines $l_1$ and $l_2$, then the linear span $\Lambda$ of $C$ in $\PP^6$ would contain both $l_1$ and $l_2$.
Consequently,
\[
 \Lambda \cap V_5\supseteq C\cup l_1\cup l_2
\]
would have degree at least $6$, contradicting the fact that $\deg(V_5\cap \Lambda)=5$.

Note that the family of $4$-dimensional linear subspaces containing a fixed line in $\PP^6$ is isomorphic to $\Gr(3,5)$, and $\bH_1(V_5)\cong \PP^2$. Hence $\bH_4(V_5)$ is isomorphic to a $\Gr(3,5)$-bundle over $\PP^2$ (\cite[Theorem 1.1]{CKK25b}).
A key technical step in the proof of this claim was to exclude the possibility that a non-Cohen--Macaulay curve $C$ with Hilbert polynomial $\chi(\cO_C(m))=4m+1$ is supported on an elliptic quartic curve (\cite[Section~4.2]{CKK25b}).
This was achieved by studying the geometry of surfaces arising as hyperplane sections of the quintic del Pezzo variety $V_5$.

\begin{proposition}[{\cite[Corollary 2.1]{FN89}, \cite[1.2.1]{Ili94}}]\label{recall}
There exists a curve $\Sigma$ and a surface $B$ containing $\Sigma$ on $V_5$ which satisfy the following.
\begin{align*}
\Sigma&=\{x\in V_5\mid \text{$\exists$ only one line $\ell$ in $V_5$ such that $x\in\ell$}\},\\
B\setminus \Sigma &=\{x\in V_5\mid \text{$\exists$ exactly two lines $\ell$ in $V_5$ such that $x\in\ell$}\}\\
V_5\setminus B&=\{x\in V_5\mid \text{$\exists$ exactly three lines $\ell$ in $V_5$ such that $x\in\ell$}\}
\end{align*}
Moreover, $B\in|-K_{V_5}|=|2H^*|$ where $H^*=\Oo_{V_5}(1)$.
\end{proposition}

\subsection{Sarkisov Link of $V_5$ (\cite{IP99})}\label{susar}

The quintic del Pezzo threefold $V_5$ is birationally equivalent to the prime Fano threefold $V_{22}$ of degree $22$ via a Sarkisov link.
This link is realized by the double projection of $V_{22}$ from a line in $\PP^{13}$.
\[\xymatrix{
& E'\subseteq V'\supseteq D' \ar[ld]_f \ar[rd]^{\varphi'} \ar@{<-->}^{\text{$(H'-D')$-flop $\chi$}}[rr] &&
E^+\subseteq V^+\supseteq D^+ \ar[ld]_{\varphi^+} \ar[rd]^g
& \\
\PP^{13}\supseteq V_{22}\supseteq L \ar@{-->}[rr] &&
Y\subseteq\PP^{11} \ar@{-->}[rr] &&
\Gamma\subseteq V_5\subseteq\PP^6
}\]
Let $H=\Oo_{V_{22}}(1)$ and $H^*=\Oo_{V_5}(1)$.
We also denote by $H'=f^*H$ and $H^+=g^*H^*$.
Then
\[
\chi^{-1}(H^+)=H'-2D'
\quad\text{and}\quad
\chi(H)=3H^+-2E^+
\]
where $L$ is a line in $V_{22}$, $\Gamma$ is a smooth rational curve of degree $5$ in $V_5$, $D'$ is the $f$-exceptional divisor, and $D^+=\chi(D')$ satisfies
$D^+=H^+-E^+$ and $g(D^+)\in |H^*-\Gamma|$ (\cite[Remark~5.2.4]{KPS18}).
Moreover, $E^+$ is the $g$-exceptional divisor, and $E'=\chi^{-1}(E^+)$ satisfies
$E'=H'-3D'$ and $f(E')\in |H-3L|$ (\cite[Remark~5.2.4]{KPS18}).
Furthermore, $f(E')$ is the ruled surface swept out by conics meeting $L$ in $V_{22}$ (\cite[Remark~4.3.4]{IP99}).
The flopping locus in $V'$ is the union of the strict transforms of lines on $V_{22}$ intersecting $L$, together with the exceptional section of $D'$ when the line $L$ is special.
Similarly, the flopping locus in $V^+$ is the union of the strict transforms of bisecant lines of $\Gamma$ on $V_5$ (\cite[Lemma~5.2.5]{KPS18}).

\begin{notation}\label{defA}
We denote by $A$ the unique hyperplane section $A\in|H^*-\Gamma|$. In the notation of Section~2.2, this hyperplane section is given by $A=g(D^+)$.
\end{notation}

Note that the image of $A$ under the Sarkisov link ($f\circ \chi \circ g^{-1}$) is the line $L$ in $V_{22}$.

Let $C$ be a degree $d$ curve in $V_{22}$ which does not intersect either $L$ or the flopping locus of $\chi$, and we denote by $\overline{C}$ (resp. $C'$, $C^+$) the strict transform of $C$ on $V_5$ (resp. $V'$, $V^+$).
Then the degree of $\overline{C}$ on $V_5$ is given as follows.
\[
H^*.\overline{C}
=H^+.C^+
=\chi^{-1}(H^+).C'
=(H'-2D').C'
=H.C=d.
\]
Moreover, the image $\overline{C}$ intersects $\Gamma$ in $V_5$ at $d$ points since
\[
\Gamma.\overline{C}
=E^+.C^+
=\chi^{-1}(E^+).C'
=(H'-3D').C'
=H.C=d.
\]

\section{Rational quartic curves in $V_{22}$}

The Sarkisov link between $V_5$ and $V_{22}$ induces a relationship between certain families of rational curves in these varieties.
We first study the geometry of the quintic del Pezzo threefold $V_5$ arising from its Pl\"ucker embedding $V_5 \subset \PP^6$.
In this section, we will repeatedly denote by $p_1$, $p_2$, $p_3$, and $p_4$ four points of $V_5$, which are also points of a rational quintic curve $\Gamma$ in $V_5$. They correspond to a point of the fourth symmetric product of the curve $\Gamma$, and they span a linear subspace in $\PP^6$.
We denote by
\[
D=p_1+p_2+p_3+p_4\in \Sym^4\Gamma
\quad\text{and}\quad
\Lambda_D=\langle p_1,\,p_2,\,p_3,\,p_4\rangle\subseteq\PP^6.
\]
We say that the four points $p_1$, $p_2$, $p_3$, and $p_4$ are general if $D\in \Sym^4\Gamma$ is general and their linear span $\Lambda_D$ has dimension $3$.

\subsection{Lemmas}
We begin with several lemmas needed for the proof of Theorem~\ref{mainthm} ($=$ Theorem~\ref{double_cover}).

\begin{lemma}\label{incidence}
Let $\Gamma$ be a rational quintic curve in $V_5$, and $p_1$, $p_2$, $p_3$, and $p_4$ be four general points of $\Gamma$ in $V_5$.
Then the linear subspace $\Lambda_D$ spanned by these four points does not contain any line in $V_5$.
\end{lemma}
\begin{proof}
Consider the incidence
\[
\Ii=\{(D,\ell)\in \Sym^4\Gamma\times \bH_1(V_5)\mid \ell\subseteq \Lambda_D\}
\]
together with the two projection maps: $\pi_1: \Ii\to \Sym^4\Gamma$ and $\pi_2: \Ii\to\bH_1(V_5)$.
We claim that $\Ii$ has dimension $3$.
Since $\bH_1(V_5)$ has dimension $2$, it suffices to show that $\pi_2^{-1}(\ell)$ has dimension at most $1$ for all $\ell \in \bH_1(V_5)$.

Fix a line $\ell$ in $V_5$, and consider a linear projection $P_{\ell}: \PP^6\dashrightarrow \PP^4$ from $\ell$.
Then the condition $\ell\subseteq\Lambda_D$ is equivalent to $\dim P_{\ell}(\Lambda_D)=1$, that is, the points $P_{\ell}(x_1)$, $P_{\ell}(x_2)$, $P_{\ell}(x_3)$, $P_{\ell}(x_4)$ are collinear in $\PP^4$.
However, the dimension of the family of $4$-secant lines of the nondegenerate curve $P_{\ell}(\Gamma)$ does not exceed $1$, so the proof is completed.
\end{proof}

\begin{lemma}\label{three_quartics}
Let $p_1$, $p_2$, $p_3$, and $p_4$ be four general points of the curve $\Gamma$ in $V_5$.
Then there exist exactly three rational quartic curves in $V_5$ passing through these four points.
\end{lemma}

\begin{proof}
Let $\Lambda_D$ be the $3$-dimensional linear subspace of $\PP^6$ spanned by $p_1$, $p_2$, $p_3$, and $p_4$.
Then $\Lambda_D$ intersects $V_5$ in five points.
We denote by $p_5$ the fifth point of $V_5\cap \Lambda_D$ other than $p_1$, $p_2$, $p_3$, and $p_4$, and define the residual map
\begin{equation*}\label{residual_map}
\rho:\Sym^4\Gamma\dashrightarrow V_5;\ p_1+p_2+p_3+p_4\mapsto p_5.
\end{equation*}

Let $\ell$ be a line in $V_5$ passing through the point $p_5$.
By Lemma~\ref{incidence}, $\Lambda_D$ does not contain $\ell$.
Let $\overline{\Lambda}_{D,\ell}=\langle\Lambda_D,\ell\rangle$ be the $4$-dimensional linear subspace of $\PP^6$ spanned by $\Lambda_D$ and $\ell$.
Then the intersection $V_5\cap \overline{\Lambda}_{D,\ell}$ is a reducible curve of degree~$5$ of the form $\ell\cup C$, and hence $C$ is a curve of degree~$4$.
By adjunction, the curve $\ell\cup C$ has arithmetic genus~$1$.
Since $\ell$ is a secant line to $C$, it follows that $C$ is a rational quartic curve.
Note that there are three distinct choices for the line $\ell=\ell_j$.
Consequently, there exist three distinct rational quartic curves $C=C_j$, each associated with a different line $\ell_j$ for $j=1,2,3$.
\end{proof}

\begin{remark}
The result in Lemma~\ref{three_quartics} coincides with the \emph{correlator} computed in Section~2.2 of \cite{BM04}.
\end{remark}

\begin{lemma}\label{quartic_in_dPz}
Let $p_1$, $p_2$, $p_3$, and $p_4$ be four points of the curve $\Gamma$ in general position.
Then there exists a unique rational quartic curve contained in the hyperplane section $A$ (Notation~\ref{defA}) passing through these four points.
\end{lemma}

\begin{proof}
We prove the lemma case by case according to the classification of the hyperplane section $A$ of $V_5$ given in \cite[Section 2]{Pro92}.

Case 1.
$A$ is non-normal, and its normalization $\phi\colon S\to A$ is isomorphic to the third Hirzebruch surface $\FF_3$.
Let $s$ and $f$ denote the exceptional section and a fiber, respectively, and set $\overline{\Gamma}=\phi^{-1}(\Gamma)$.
Then $\overline{\Gamma}\in |s+4f|$ and $C\in|s+3f|$.
Note that $\dim|s+3f|=\dim\PP H^0(\PP^1,\Oo_{\PP^1}\oplus\Oo_{\PP^1}(3))=4$, hence four general points determine a unique member.

Case 2.
$A$ is non-normal, and its normalization $\phi\colon S\to A$ is isomorphic to the first Hirzebruch surface $\FF_1$.
Let $s$ and $f$ denote the exceptional section and a fiber, respectively, and set $\overline{\Gamma}=\phi^{-1}(\Gamma)$.
Then $\overline{\Gamma}\in |s+3f|$ and $C\in|s+2f|$.
Note that $\dim|s+2f|=\dim\PP H^0(\PP^1,\Oo_{\PP^1}(1)\oplus\Oo_{\PP^1}(2))=4$, hence four general points determine a unique member.

Case 3.
$A$ is normal, and its desingularization $\phi: S\to A$ is isomorphic to a del Pezzo surface of degree~$5$.
Let $\overline{\Gamma}=\phi^{-1}(\Gamma)$.
By identifying $S$ with the blow-up $\psi:S\to\mathbb{P}^2$ of four points $q_1,\ldots,q_4\in\mathbb{P}^2$, we may assume that $\Gamma_0:=\psi(\overline{\Gamma})$ is a smooth conic on $\PP^2$.
Since $-K_S\cdot\overline{\Gamma}=5$, the curve $\Gamma_0$ passes through exactly one of $q_1,\,q_2,\,q_3,\,q_4$.
After relabeling, we may assume that $q_1\in\Gamma_0$ and $q_2,q_3,q_4\notin\Gamma_0$.

In addition, the existence of a curve $C$ on $S$ with $-K_S\cdot C=4$ and $x_1,\ldots,x_4\in C$ is equivalent to the existence of a nodal cubic $C_0\subset \mathbb{P}^2$ passing through $q_1,\ldots,q_4$ and $\psi(x_1),\ldots,\psi(x_4)$, with node at $q_1$, and $q_1,\psi(x_1),\ldots,\psi(x_4)$ lie on the same conic $\Gamma_0$.
For general choices of the points, they determine a unique plane cubic $C_0$.
\end{proof}

\subsection{Proof of Theorem~\ref{mainthm}}
Let $V_{22}$ be a prime Fano threefold of degree $22$.
Assume that the Hilbert scheme $\bH_4(V_{22})$ is irreducible.
\begin{theorem}\label{double_cover}
Under the above assumption, the Hilbert scheme $\bH_4(V_{22})$ is birational to a double cover of $\PP^4(\cong\Sym^4\Gamma)$ branched over a hypersurface.
\end{theorem}

\begin{proof}
We denote by $\bH_4(V_{5},\Gamma)$ the Hilbert scheme of rational quartic curves $C$ in $V_{5}$ such that the $0$-dimensional subscheme $C\cap \Gamma$ has length $4$.
By Lemma~\ref{three_quartics}, there exists a generically $3$-to-$1$ rational map
\[
\bH_4(V_{5},\Gamma)
\dashrightarrow
\Sym^4\Gamma\cong \PP^4.
\]
Let $[C]\in \bH_4(V_{5},\Gamma)$, and assume that $C$ admits a strict transform $\overline{C}$ on $V_{22}$ under the Sarkisov link.
Then
\[
\overline{C}.H
=C.(3H^*-2\Gamma)
=3\cdot4-2\cdot4
=4,
\]
so $\overline{C}$ is a rational quartic curve in $V_{22}$.
Note that $C$ admits a strict transform on $V_{22}$ if and only if $C\not\subseteq A$.
Let us define
\[
\mathbf{S}_A
=\{[C]\in \bH_4(V_5,\Gamma)\mid C\subseteq A\},
\]
which parametrizes the rational quartic curves $C$ in $V_5$ such that $C\cap \Gamma$ is a $0$-dimensional subscheme of length $4$ and the curve $C$ is contracted by the Sarkisov link to $V_{22}$.
Then $\mathbf{S}_A$ maps birationally onto $\Sym^4\Gamma$ by Lemma~\ref{quartic_in_dPz}.
Consequently, the map
\[
\bH_4(V_{22})
\dashrightarrow
\bH_4(V_5,\Gamma)\setminus \mathbf{S}_A
\]
is birational, and hence we obtain a generically $2$-to-$1$ rational map
\[
\Phi:
\bH_4(V_{22})
\dashrightarrow
\Sym^4\Gamma\cong \PP^4.
\]

We now describe the branch locus of $\Phi$.
Recall that exactly three lines pass through a general point $x\in V_5$.
Let $x_1$, $x_2$, $x_3$, and $x_4$ be four general points of $\Gamma$, and let $\ell_1$, $\ell_2$, and $\ell_3$ be the lines passing through the residual point $x_5=\rho(x_1+x_2+x_3+x_4)\in V_5$.
As in the proof of Lemma~\ref{three_quartics}, define
\[
\overline{\Lambda}_i
=\langle x_1,x_2,x_3,x_4,\ell_i\rangle \subseteq \PP^6.
\]
Then the residual curves $C_i$ to $\ell_i$ in $\overline{\Lambda}_i\cap V_5$ are precisely the preimages of the map $\Phi$ over the point $x_1+x_2+x_3+x_4\in \Sym^4\Gamma$.
Also, the map $\Phi$ branches precisely when two of the three associated quartic curves $C_i$ coincide.
This occurs exactly when two of the lines among $\ell_1$, $\ell_2$, and $\ell_3$ coincide and the remaining line is contained in $A$.
Equivalently, after relabeling,
\[
\ell_1=\ell_2
\qquad\text{and}\qquad
\ell_3\subseteq A.
\]
Recall that the locus $B\subset V_5$ consisting of points through which two of the three lines coincide is a divisor in the linear system $|\Oo_{V_5}(2)|$ (see Proposition~\ref{recall}).
Therefore, the inverse image $\rho^{-1}(B)$ contains the branch divisor of $\Phi$. Since $\Phi$ is generically two-to-one, this branch divisor is a hypersurface.
\end{proof}

\bibliographystyle{alpha}
\bibliography{Library.bib}

\end{document}